\documentclass[a4paper, 12pt]{amsart}

\usepackage{amsmath}
\usepackage{amssymb}
\usepackage{ascmac}
\usepackage{amsthm}

\makeatletter

\@addtoreset{equation}{section}
\makeatother 

\theoremstyle{plain}
\newtheorem{thm}{Theorem}[section]
\newtheorem*{thm*}{Theorem}
\newtheorem{prop}[thm]{Proposition}
\newtheorem{lem}[thm]{Lemma}
\newtheorem{cor}[thm]{Corollary}

\theoremstyle{definition}
\newtheorem{defn}[thm]{Definition}

\theoremstyle{remark}
\newtheorem{rem}[thm]{Remark}

\newcommand{\vol}{\operatorname{vol}}
\newcommand{\ric}{\operatorname{Ric}}

\newcommand{\Hess}{\operatorname{Hess}}

\newcommand{\tr}{\operatorname{trace}}

\newcommand{\cut}{\mathrm{Cut}\,}

\newcommand{\supp}{\mathrm{supp}\,}

\newcommand{\CD}{\operatorname{CD}}

\usepackage{latexsym}

\title[One dimensional weighted Ricci curvature]{One dimensional weighted Ricci curvature and displacement convexity of entropies}

\author{Yohei Sakurai}
\address{Advanced Institute for Materials Research, Tohoku University, 2-1-1 Katahira, Aoba-ku, Sendai, 980-8577, Japan}
\email{yohei.sakurai.e2@tohoku.ac.jp}

\date{May 25, 2020}

\begin{document}
\maketitle

\begin{abstract}
In the present paper,
we prove that
a lower bound on the $1$-weighted Ricci curvature is equivalent to a convexity of entropies on the Wasserstein space.
Based on such characterization,
we provide some interpolation inequalities such as the Pr\'ekopa-Leindler inequality,
the Borel-Branscamp-Lieb inequality,
and the Brunn-Minkowski inequality under the curvature bound.
\end{abstract}

\section{Introduction}\label{sec:Introduction}
The aim of this article is to characterize a lower bound on the $1$-weighted Ricci curvature in terms of a convexity of entropies on the Wasserstein space.
Such characterization enables us to produce various interpolation inequalities under the lower bound on the $1$-weighted Ricci curvature.

For $n\geq 2$,
let $(M,d,m)$ be an $n$-dimensional weighted Riemannian manifold,
namely,
$M=(M,g)$ is an $n$-dimensional complete Riemannian manifold (without boundary),
$d$ is the Riemannian distance on $M$,
and $m:=e^{-f}\,\vol$ for some $f\in C^{\infty}(M)$,
where $\vol$ denotes the Riemannian volume measure on $M$.
For $N\in (-\infty,\infty]$,
the \textit{$N$-weighted Ricci curvature} $\ric^{N}_{f}$ is defined as follows (\cite{Ba}, \cite{BE}, \cite{L}, \cite{Q}):
If $N \in (-\infty,\infty)\setminus \{n\}$,
then
\begin{equation*}\label{eq:def of weighted Ricci curvature}
\ric^{N}_{f}:=\ric_{g}+\Hess f-\frac{d f \otimes d f}{N-n},
\end{equation*}
where $\ric_{g}$ is the Ricci curvature determined by $g$,
and $d f$ and $\Hess f$ are the differential and the Hessian of $f$,
respectively;
otherwise,
if $N=\infty$,
then $\ric^{N}_{f}:=\ric_{g}+\Hess f$;
if $N=n$,
and if $f$ is constant,
then $\ric^{N}_{f}:=\ric_{g}$;
if $N=n$,
and if $f$ is not constant,
then $\ric^{N}_{f}:=-\infty$.
For $\mathcal{F}:M\to \mathbb{R}$,
we mean by $\ric^{N}_{f}\geq \mathcal{F}$
for every $x \in M$,
and for every unit tangent vector $v$ at $x$
we have $\ric^{N}_{f}(v)\geq \mathcal{F}(x)$.
Traditionally,
the parameter $N$ has been chosen from $[n,\infty]$,
and in that case,
we have already known many geometric and analytic properties (see e.g., \cite{Lo}, \cite{LV1}, \cite{LV2}, \cite{Q}, \cite{St1}, \cite{St2}, \cite{St3}, \cite{WW}).
On the other hand,
very recently,
in the complementary case of $N\in (-\infty,n)$,
various properties have begun to be studied beyond the traditional case (see e.g., \cite{K}, \cite{KM}, \cite{Mi}, \cite{O2}, \cite{O3}, \cite{OT1}, \cite{OT2}, \cite{S}, \cite{W}, \cite{WY}).

It is well-known that
lower bounds on the $N$-weighted Ricci curvature can be characterized by convexities of entropies on the Wasserstein space via optimal transport theory.
We consider a curvature condition
\begin{equation}\label{eq:constant Ricci curvature assumption}
\ric^{N}_{f} \geq K
\end{equation}
for $K\in \mathbb{R}$.
In the traditional case of $N\in [n,\infty]$,
the characterization of the curvature condition (\ref{eq:constant Ricci curvature assumption}) is due to von Renesse and Sturm \cite{RS}, and Sturm \cite{St1} for $N=\infty$,
and Sturm \cite{St2}, \cite{St3}, and Lott and Villani \cite{LV1}, \cite{LV2} for $N\in [n,\infty]$.
Based on such characterization results,
for general metric measure spaces,
Sturm \cite{St2}, \cite{St3}, and Lott and Villani \cite{LV1}, \cite{LV2} have independently introduced the so-called \textit{curvature-dimension-condition} $\CD(K,N)$ for $K\in \mathbb{R}$ and $N\in [1,\infty]$.
The curvature-dimension condition $\CD(K,N)$ is equivalent to the condition (\ref{eq:constant Ricci curvature assumption}) when $N\in [n,\infty]$ on weighted Riemannian manifolds,
and it has been studied from various perspectives.

In the complementary case of $N\in (-\infty,n)$,
Ohta \cite{O2} has recently characterized the condition (\ref{eq:constant Ricci curvature assumption}) for $N\in (-\infty,0)$.
Based on the characterization,
he further formulated the curvature-dimension condition $\CD(K,N)$ for $K\in \mathbb{R}$ and $N\in (-\infty,0)$ (see also earlier works done by Ohta and Takatsu \cite{OT1}, \cite{OT2}).
Ohta \cite{O3} has also extended this program to the case of $N=0$.

Now,
we are concerned with the characterization problem of lower bounds on the $N$-weighted Ricci curvature in the remaining case of $N\in (0,n)$.
We focus on the case of $N=1$,
especially a condition
\begin{equation}\label{eq:Ricci curvature assumption}
\ric^{1}_{f}\geq (n-1)\, \kappa\, e^{\frac{-4f}{n-1}}
\end{equation}
for $\kappa \in \mathbb{R}$ introduced by Wylie and Yeroshkin \cite{WY} from the view point of the study of weighted connection.
They have observed that
the curvature condition (\ref{eq:Ricci curvature assumption}) is equivalent to a lower Ricci curvature bound by $(n-1)\kappa$ with respect to weighted connection (see Proposition \ref{prop:interpretation of curvature condition} below).
They further established comparison geometry under the condition (\ref{eq:Ricci curvature assumption}) (more precisely, see Subsection \ref{sec:Geometric analysis}).
In this paper,
inspired by the interpretation of the curvature condition (\ref{eq:Ricci curvature assumption}) via weighted connection,
we will prove that
the condition (\ref{eq:Ricci curvature assumption}) can be characterized by a convexity of entropies on the Wasserstein space.
Based on the equivalence,
we conclude interpolation inequalities such as the Pr\'ekopa-Leindler inequality,
the Borel-Branscamp-Lieb inequality,
and the Brunn-Minkowski inequality.

\subsection{Main result}
To state our main result,
we introduce a convexity property of entropies on the Wasserstein space (see Subsection \ref{sec:Optimal transport theory} for basics of optimal transport theory, and precise definition of the Wasserstein space).
Hereafter,
let $(M,d,m)$ be an $n$-dimensional weighted Riemannian manifold,
where $m=e^{-f} \,\vol$ for some $f\in C^{\infty}(M)$.

Let $P_{2}(M)$ denote the set of all Borel probability measures $\mu$ on $M$ satisfying
\begin{equation*}
\int_{M}\,d(x,x_{0})^{2}\,d\mu(x)<\infty
\end{equation*}
for some point $x_{0} \in M$,
which admits a metric called \textit{$L^{2}$-Wasserstein distance function} $W_{2}$ (see Subsection \ref{sec:Optimal transport theory} for its precise definition).
The metric space $(P_{2}(M),W_{2})$ called \textit{$L^{2}$-Wasserstein space} is one of the most fundamental objects in optimal transport theory.
Let $\mathcal{DC}$ stand for the set of all continuous convex functions $U:[0,\infty)\to \mathbb{R}$ with $U(0)=0$ such that
a function $\varphi_{U}:(0,\infty)\to \mathbb{R}$ defined by $\varphi_{U}(r):=r^{n}\,U(r^{-n})$ is convex.
For $U\in \mathcal{DC}$,
an entropy functional $U_{m}$ on $P_{2}(M)$ is defined by
\begin{equation}\label{eq:entropy}
U_{m}(\mu):=\int_{M}\,U(\rho)\,dm,
\end{equation}
where $\rho$ is the density of the absolutely continuous part in the Lebesgue decomposition of $\mu$ with respect to $m$.
For a function $H \in \mathcal{DC}$ defined by $H(r):=n\,r(1-r^{-\frac{1}{n}})$,
the entropy functional $H_{m}$ on $P_{2}(M)$ defined as (\ref{eq:entropy}) is called the \textit{R\'enyi entropy}.

In order to introduce our convexity property of entropies,
we need to define the twisted coefficient.
For $t \in [0,1]$,
we consider two lower semi-continuous functions $d_{f,t},\,d_{f}:M\times M \to \mathbb{R}$ defined by
\begin{equation*}\label{eq:weighted distance function}
d_{f,t}(x,y):=\inf_{\gamma} \int^{t\,d(x,y)}_{0}\,e^{\frac{-2f(\gamma(\xi))}{n-1}}\,d\xi,\quad d_{f}:=d_{f,1},
\end{equation*}
where the infimum is taken over all unit speed minimal geodesics $\gamma:[0,d(x,y)]\to M$ from $x$ to $y$.
The function $d_{f}$ has been called the \textit{re-parametrize distance} in \cite{WY} (cf. Subsection \ref{sec:Geometric analysis}).
In the unweighted case of $f=0$,
we have $d_{f,t}=t\,d$.
Notice that
for $t\in (0,1]$,
the function $d_{f,t}$ is not necessarily distance
since the triangle inequality does not hold in general.
We also remark that
for $t\in (0,1)$,
the function $d_{f,t}$ is not always symmetric.
For $\kappa \in \mathbb{R}$,
let $\mathfrak{s}_{\kappa}(t)$ be a unique solution of the Jacobi equation $\psi''(t)+\kappa\, \psi(t)=0$ with $\psi(0)=0,\,\psi'(0)=1$,
and let $C_{\kappa}$ be the diameter of the space form of constant curvature $\kappa$;
more explicitly,
they are written as
\begin{equation*}
\mathfrak{s}_{\kappa}(t)=\begin{cases}
                                                     \displaystyle \frac{\sin \sqrt{\kappa}t}{\sqrt{\kappa}} & \text{if $\kappa>0$}, \\
                                                                                t           & \text{if $\kappa=0$},\\
                                                     \displaystyle \frac{\sinh \sqrt{\vert\kappa \vert}t}{\sqrt{\vert \kappa \vert}}           & \text{if $\kappa<0$},
                                                   \end{cases}\quad 
C_{\kappa}=\begin{cases}
                                                     \displaystyle \frac{\pi}{\sqrt{\kappa}} & \text{if $\kappa>0$}, \\
                                                                                \infty           & \text{if $\kappa \leq 0$}.
                                                   \end{cases}
\end{equation*}
For $t\in (0,1)$ and $x,y\in M$ with $x\neq y$,
we define the \textit{twisted coefficient} $\beta_{\kappa,f,t}(x,y)$ as follows:
\begin{equation*}\label{eq:weighted distorsion coefficient}
\beta_{\kappa,f,t}(x,y):=\left(  \frac{\mathfrak{s}_{\kappa}( d_{f,t}(x,y))}{t\, \mathfrak{s}_{\kappa}(d_{f}(x,y))}  \right)^{n-1}
\end{equation*}
if $d_{f}(x,y)\in (0,C_{\kappa})$;
otherwise,
$\beta_{\kappa,f,t}(x,y):=\infty$.

\begin{rem}\label{rem:disjoint}
In the unweighted case of $f=0$,
we can define the twisted coefficient for $x=y$ as the limit $1$.
\end{rem}

Let $P^{ac}_{2}(M)$ denote the set of all Borel probability measures in $P_{2}(M)$ that are absolutely continuous with respect to $m$.
Let us introduce the following notion:
\begin{defn}\label{defi:twisted curvature bound}
For $\kappa \in \mathbb{R}$,
we say that
$(M,d,m)$ has \textit{$\kappa$-twisted curvature bound} if
for every disjointly supported pair $\mu_{0},\mu_{1} \in P^{ac}_{2}(M)$,
there are an optimal coupling $\pi$ of $(\mu_{0},\mu_{1})$,
and a minimal geodesic $(\mu_{t})_{t\in [0,1]}$ in the $L^{2}$-Wasserstein space from $\mu_{0}$ to $\mu_{1}$ such that
for all $U\in \mathcal{DC}$ and $t\in (0,1)$,
\begin{align}\label{eq:lower twisted curvature bound}
U_{m}(\mu_{t}) \leq (1-t)\,&\int_{M\times M}\,       U\left(  \frac{\rho_{0}(x)}{ \beta_{\kappa,f,1-t}(y,x) }  \right)  \frac{\beta_{\kappa,f,1-t}(y,x)}{\rho_{0}(x)}\, d\pi(x,y)\\
                                     +\:t\,&\int_{M\times M}\,       U\left(  \frac{\rho_{1}(y)}{\beta_{\kappa,f,t}(x,y)}  \right)    \frac{\beta_{\kappa,f,t}(x,y)}{\rho_{1}(y)}\, d\pi(x,y), \notag
\end{align}
where $\rho_{i}$ is the density of $\mu_{i}$ with respect to $m$ for each $i=0,1$.
\end{defn}

We also introduce the following weaker version:
\begin{defn}\label{defi:weak twisted curvature bound}
For $\kappa \in \mathbb{R}$,
we say that
$(M,d,m)$ has \textit{$\kappa$-relaxed twisted curvature bound} if
for every disjointly supported pair $\mu_{0},\mu_{1} \in P^{ac}_{2}(M)$,
there exist an optimal coupling $\pi$ of $(\mu_{0},\mu_{1})$,
and a minimal geodesic $(\mu_{t})_{t\in [0,1]}$ in the $L^{2}$-Wasserstein space from $\mu_{0}$ to $\mu_{1}$ such that
for $H \in \mathcal{DC}$ defined as $H(r):=n\,r(1-r^{-\frac{1}{n}})$,
and for every $t\in (0,1)$
the inequality (\ref{eq:lower twisted curvature bound}) holds.
\end{defn}

\begin{rem}
In the unweighted case where $f=0$,
the notion of the $\kappa$-twisted curvature bound coincides with that of the curvature-dimension condition $\CD((n-1)\kappa,n)$ in the sense of Lott and Villani \cite{LV1}, \cite{LV2} (except for the disjointness of $\mu_{0},\mu_{1}$ in view of Remark \ref{rem:disjoint}).
Similarly,
the notion of the $\kappa$-relaxed twisted curvature bound coincides with that of the curvature-dimension condition $\CD((n-1)\kappa,n)$ in the sense of Sturm \cite{St2}, \cite{St3}.
\end{rem}

Our main result is the following characterization theorem:
\begin{thm}\label{thm:displacement convexity}
Let $(M,d,m)$ be an $n$-dimensional weighted Riemannian manifold,
where $m:=e^{-f} \,\vol$ for some $f\in C^{\infty}(M)$.
Let $\kappa \in \mathbb{R}$.
Then the following statements are equivalent:
\begin{enumerate}\setlength{\itemsep}{+0.7mm}
\item $\ric^{1}_{f} \geq (n-1)\kappa\,e^{\frac{-4f}{n-1}}$; \label{enum:curv cond}
\item $(M,d,m)$ has $\kappa$-twisted curvature bound; \label{enum:twisted curv}
\item $(M,d,m)$ has $\kappa$-relaxed twisted curvature bound. \label{enum:relaxed twisted curv}
\end{enumerate}
\end{thm}

For $K \in \mathbb{R}$ and $N \in [n,\infty]$,
Lott and Villani \cite{LV1} have characterized the curvature condition (\ref{eq:constant Ricci curvature assumption}) by a convexity of entropies on the Wasserstein space (see Theorem 4.22 in \cite{LV1}).
The Lott-Villani theorem in a special case where $f=0,\,K=(n-1)\kappa$ and $N=n$ states that
the statements \ref{enum:curv cond} and \ref{enum:twisted curv} in Theorem \ref{thm:displacement convexity} are equivalent when $f=0$.

\begin{rem}
In \cite{LV1},
they have also investigated the $1$-weighted Ricci curvature (see Definition 4.20 in \cite{LV1}).
But they have defined it as $-\infty$,
which is completely different from our definition.
\end{rem}

For $K \in \mathbb{R}$ and $N \in [n,\infty)$,
Sturm \cite{St3} has characterized a condition that $\ric_{g} \geq K$ and $n\leq N$ (see Theorem 1.7 in \cite{St3}),
where $\ric_{g} \geq K$ means that
for every $x \in M$,
and for every unit tangent vector $v$ at $x$
we have $\ric_{g}(v)\geq K$.
The Sturm theorem in the special case where $K=(n-1)\kappa$ and $N=n$ tells us that
the statements \ref{enum:curv cond} and \ref{enum:relaxed twisted curv} in Theorem \ref{thm:displacement convexity} are equivalent when $f=0$.

One of the key ingredients of the proof of Theorem \ref{thm:displacement convexity} is to obtain inequalities for Jacobians of optimal transport maps
that are associated with $\ric^{1}_{f}$.
We first show an inequality of Riccati type (see Lemma \ref{lem:Riccati inequality}).
From the inequality of Riccati type,
we derive an inequality concerning the concavity of the Jacobians under the curvature condition (\ref{eq:Ricci curvature assumption}) (see Proposition \ref{prop:Jacobian inequality}).
By using the concavity,
we prove that
the curvature condition (\ref{eq:Ricci curvature assumption}) implies the convexity of entropies.

\subsection{Organization}
In Section \ref{sec:Preliminaries},
we review the works done by Wylie and Yeroshkin \cite{WY},
and also recall basics of the optimal transport theory.
In Section \ref{sec:Key inequalities},
we show key inequalities for the proof of Theorem \ref{thm:displacement convexity}.
In Section \ref{sec:Displacement convexity},
we prove Theorem \ref{thm:displacement convexity}.
Furthermore,
under the curvature condition (\ref{eq:Ricci curvature assumption}),
we conclude various interpolation inequalities (see Subsection \ref{sec:Interpolation inequalities}).
In Section \ref{sec:Applications},
we discuss the possibility of deriving functional inequalities from Theorem \ref{thm:displacement convexity}.

\section{Preliminaries}\label{sec:Preliminaries}

\subsection{Geometric analysis on $1$-weighted Ricci curvature}\label{sec:Geometric analysis}
In this subsection,
we briefly recall the work done by Wylie and Yeroshkin \cite{WY} concerning the curvature condition (\ref{eq:Ricci curvature assumption}).
They have suggested a new approach to investigate geometric properties of weighted manifolds.
Let $\nabla$ be the Levi-Civita connection induced from $g$,
and let $\alpha$ be a $1$-form on $M$.
The basic tool in \cite{WY} was the \textit{weighted connection}
\begin{equation*}
\nabla^{\alpha}_{\mathcal{U}} \mathcal{V}:=\nabla_{\mathcal{U}} \mathcal{V}-\alpha(\mathcal{U})\mathcal{V}-\alpha(\mathcal{V})\mathcal{U}
\end{equation*}
which is torsion free,
affine,
and projectively equivalent to $\nabla$.
They have studied weighted manifolds in view of this weighted connection.

\begin{rem}
We recall that
two torsion free, affine connections $\widehat{\nabla},\overline{\nabla}$ are said to be \textit{projectively equivalent} if
they possess the same geodesics up to re-parametrization.
Due to Weyl \cite{We},
it is well-known that
they are projectively equivalent if and only if there is a $1$-form $\widehat{\alpha}$ such that
\begin{equation*}
\widehat{\nabla}_{\mathcal{U}} \mathcal{V}=\overline{\nabla}_{\mathcal{U}} \mathcal{V}+\widehat{\alpha}(\mathcal{U})\mathcal{V}+\widehat{\alpha}(\mathcal{V})\mathcal{U}.
\end{equation*}
\end{rem}

They have examined the relation between the $1$-weighted Ricci curvature and the Ricci curvature induced from $\nabla^{\alpha}$.
The \textit{$\nabla^{\alpha}$-curvature tensor} and the \textit{$\nabla^{\alpha}$-Ricci tensor} are defined as
\begin{align*}\notag
R^{\nabla^{\alpha}}(\mathcal{U},\mathcal{V})\mathcal{W}&:=\nabla^{\alpha}_{\mathcal{U}}\nabla^{\alpha}_{\mathcal{V}}\mathcal{W}-\nabla^{\alpha}_{\mathcal{V}}\nabla^{\alpha}_{\mathcal{U}}\mathcal{W}-\nabla^{\alpha}_{[\mathcal{U},\mathcal{V}]}\mathcal{W},\\ \label{eq:weighed affine connection Ricci tensor}
\ric^{\nabla^{\alpha}}(\mathcal{V},\mathcal{W})&:= \tr_{g} \left[\mathcal{U}\mapsto R^{\nabla^{\alpha}}(\mathcal{U},\mathcal{V})\mathcal{W}  \right],
\end{align*}
where $\tr_{g}$ denotes the trace with respect to $g$.
Let us consider a closed $1$-form $\alpha_{f}$ on $M$ defined by
\begin{equation*}
\alpha_{f}:=\frac{df}{n-1}.
\end{equation*}
The first key observation is that
$\ric^{\nabla^{\alpha_{f}}}$ coincides with the $1$-weighted Ricci tensor $\ric^{1}_{f}$ (see Proposition 3.3 in \cite{WY}).

They also investigated geodesics for $\nabla^{\alpha}$.
For $x\in M$,
we denote by $U_{x}M$ the unit tangent sphere at $x$.
For $v\in U_{x}M$,
let $\gamma_{v}:[0,\infty)\to M$ be the ($\nabla$-)geodesic with initial conditions $\gamma_{v}(0)=x$ and $\gamma'_{v}(0)=v$.
We now define a function $s_{f,v}:[0,\infty]\to [0,s_{f,v,\infty}]$ by
\begin{equation*}
s_{f,v}(t):=\int^{t}_{0}\,e^{\frac{-2f(\gamma_{v}(\xi))}{n-1}}\,d\xi,\quad s_{f,v,\infty}:=\int^{\infty}_{0}\,e^{\frac{-2f(\gamma_{v}(\xi))}{n-1}}\,d\xi.
\end{equation*}
Let $t_{f,v}:[0,s_{f,v,\infty}]\to [0,\infty]$ be the inverse function of $s_{f,v}$.
The second key observation is that
a curve $\widehat{\gamma}_{v}:[0,s_{f,v,\infty})\to M$ defined as $\widehat{\gamma}_{v}:=\gamma_{v}\circ t_{f,v}$ is a $\nabla^{\alpha_{f}}$-geodesic (cf. Proposition 3.1 in \cite{WY}).

Summarizing the above two key observations,
they have concluded the following interpretation of the curvature condition (\ref{eq:Ricci curvature assumption})
in terms of the $\nabla^{\alpha_{f}}$-Ricci curvature $\ric^{\nabla^{\alpha_{f}}}$ (see Subsection 2.1 in \cite{WY}):
\begin{prop}[\cite{WY}]\label{prop:interpretation of curvature condition}
For $\kappa \in \mathbb{R}$,
the following are equivalent:
\begin{enumerate}\setlength{\itemsep}{+1.0mm}
\item $\ric^{1}_{f}(\gamma'_{v}(t)) \geq (n-1)\kappa\,e^{\frac{-4f(\gamma_{v}(t))}{n-1}}$ for all $v\in U_{x}M$ and $t\in [0,\infty)$;
\item $\ric^{\nabla^{\alpha_{f}}}(\widehat{\gamma}'_{v}(s))\geq (n-1)\kappa$ for all $v\in U_{x}M$ and $s\in [0,s_{f,v,\infty})$.
\end{enumerate}
\end{prop}

Keeping in mind Proposition \ref{prop:interpretation of curvature condition},
they have developed comparison geometry under the curvature condition (\ref{eq:Ricci curvature assumption}).
Before their work,
Wylie \cite{W} has obtained a splitting theorem of Cheeger-Gromoll type under the condition $\ric^{N}_{f}\geq 0$ for $N \in (-\infty,1]$ (see Theorem 1.2 and Corollary 1.3 in \cite{W}).
After that
they have proved a Laplacian comparison for the distance function from a single point (see Theorem 4.4 in \cite{WY}),
a diameter comparison of Bonnet-Myers type for the re-parametrized distance $d_{f}$ (see Theorem 2.2 in \cite{WY}),
a maximal diameter theorem of Cheng type for the deformed metric $e^{\frac{-4f}{n-1}}g$ (see Theorem 2.6 in \cite{WY}),
and a volume comparison of Bishop-Gromov type for the weighted volume measure $e^{-\frac{n+1}{n-1}f} \vol$ (see Theorem 4.5 in \cite{WY}).

For later convenience,
we will review the diameter comparison.
For $x\in M$,
we denote by $d_{x}:M\to \mathbb{R}$ the distance function from $x$ defined as $d_{x}(y):=d(x,y)$.
For $v \in U_{x}M$
we set
\begin{equation*}\label{eq:weighted cut value}
\tau(v):=\sup \{\,t>0 \mid d_{x}(\gamma_{v}(t))=t\,\},\quad \tau_{f}(v):=s_{f,v}(\tau(v)).
\end{equation*}
They have obtained the following comparison for the re-parametrized distance $d_{f}$ (see Theorem 2.2 in \cite{WY}):
\begin{thm}[\cite{WY}]\label{thm:cut value estimate}
For $\kappa>0$,
if $\ric^{1}_{f} \geq (n-1)\kappa\,e^{\frac{-4f}{n-1}}$,
then for all $x \in M$ and $v \in U_{x}M$
we have
\begin{equation*}
\tau_{f}(v)\leq \frac{\pi}{\sqrt{\kappa}}.
\end{equation*}
Moreover,
for the re-parametrized distance $d_{f}$,
we have
\begin{equation*}
\sup_{x,y\in M}\,d_{f}(x,y)\leq \frac{\pi}{\sqrt{\kappa}}.
\end{equation*}
\end{thm}

We finally list some related works to \cite{WY}.
Li and Xia \cite{LX} have produced formulas of Bochner type and Reilly type with respect to $\ric^{\nabla^{\alpha}}$.
The author \cite{S} has studied comparison geometry of manifolds with boundary under the curvature condition (\ref{eq:Ricci curvature assumption}).

\subsection{Optimal transport theory}\label{sec:Optimal transport theory}
In the present subsection,
we correct some basic facts of the optimal transport theory in our setting,
which will be used in the proof of our main result.
We refer to \cite{CMS}, \cite{O1}, \cite{V},
and also the preliminaries of \cite{OT1}, \cite{OT2}.

Let $(Z,d_{Z})$ be a metric space.
A curve $\gamma:[0,l]\to Z$ is said to be a \textit{minimal geodesic} if
there exists $a \geq 0$ such that for all $t_{0},t_{1} \in [0,l]$ we have $d_{Z}(\gamma(t_{0}),\gamma(t_{1}))=a\,\vert t_{0}-t_{1} \vert$.
Moreover,
if $a=1$,
then $\gamma$ is called a \textit{unit speed minimal geodesic}.

Let $\mu,\nu$ be two Borel probability measures on $M$.
A Borel probability measure $\pi$ on $M\times M$ is said to be a \textit{coupling of $(\mu,\nu)$} if
$\pi(X\times M)=\mu(X)$ and $\pi(M\times X)=\nu(X)$ for all Borel subsets $X\subset M$.
Let $\Pi(\mu,\nu)$ denote the set of all couplings of $(\mu,\nu)$.
The \textit{$L^{2}$-Wasserstein distance function $W_{2}:P_{2}(M)\times P_{2}(M)\to [0,\infty)$} is defined as
\begin{equation}\label{eq:Wasserstein distance}
W_{2}(\mu,\nu):=\inf_{\pi \in \Pi(\mu,\nu)}\, \left(\int_{M\times M}\,d(x,y)^{2}\,d\pi(x,y)\right)^{\frac{1}{2}}.
\end{equation}
The pair $(P_{2}(M),W_{2})$ is well-known to be a complete separable metric space (see e.g., Theorem 6.18 in \cite{V}),
and called the \textit{$L^{2}$-Wasserstein space over $M$}.
A coupling $\pi \in \Pi(\mu,\nu)$ is said to be \textit{optimal} if it attains the infimum of (\ref{eq:Wasserstein distance}).

We now recall the following fundamental characterization result for the optimal coupling on Riemannian manifold (see \cite{B}, \cite{FG}, \cite{M}):
\begin{thm}\label{thm:McCann theorem}
For $\mu \in P^{ac}_{2}(M)$ and $\nu\in P_{2}(M)$,
there exists a locally semi-convex function $\phi$ on $M$ such that
a map $F_{t}$ on $M$ defined by
\begin{equation}\label{eq:interpolation map}
F_{t}(z):=\exp_{z} (t \nabla \phi(z))
\end{equation}
gives a unique optimal coupling $\pi$ of $(\mu,\nu)$ via the pushforward measure $\pi:=(F_{0}\times F_{1})_{\#}\mu$ of $\mu$ by $F_{0}\times F_{1}$,
and also determines a unique minimal geodesic $(\mu_{t})_{t\in [0,1]}$ in $(P_{2}(M),W_{2})$ from $\mu$ to $\nu$ via the pushforward measure $\mu_{t}:=(F_{t})_{\#}\mu_{0}$ of $\mu_{0}$ by $F_{t}$.
Here $\exp_{z}$ denotes the exponential map at $z$,
and $\nabla \phi$ denotes the gradient of $\phi$.
\end{thm}

Brenier \cite{B} has firstly established such characterization result in the Euclidean setting.
McCann \cite{M} has shown Theorem \ref{thm:McCann theorem} in the case where $M$ is compact (see Section 3 in \cite{M}),
and Figalli and Gigli \cite{FG} have extended it to the non-compact case (see Theorem 1 in \cite{FG}).

The locally semi-convex function $\phi$ obtained in Theorem \ref{thm:McCann theorem} is called the \textit{Kantorovich potential},
which is twice differentiable $\mu$-almost everywhere due to the Alexandrov-Bangert theorem (\cite{A}, \cite{Ban}).
The Kantorovich potential $\phi$ further enjoys the following property (see e.g., Propositions 2.5 and 4.1 in \cite{CMS}):
If $\phi$ is twice differentiable at $x$,
then for every $t\in [0,1]$,
the point $F_{t}(x)$ does not belong to the cut locus $\cut x$ of $x$,
and the differential $(dF_{t})_{x}$ of $F_{t}$ at $x$ is well-defined.
Moreover,
$\phi$ also satisfies the following property (see e.g., Proposition 5.4 in \cite{CMS}):
If $\nu$ also belongs to $P^{ac}_{2}(M)$,
then the unique minimal geodesic $(\mu)_{t\in [0,1]}$ lies in $P^{ac}_{2}(M)$.

We close this subsection with the following \textit{Jacobian equation},
which is also called the \textit{Monge-Amp\'ere equation} (see e.g., Theorem 4.2 in \cite{CMS}):
\begin{thm}\label{thm:Jacobian equation}
Let $\mu,\nu\in P^{ac}_{2}(M)$,
and let $\phi$ denote the Kantorovich potential obtained in Theorem \ref{thm:McCann theorem}.
Then for $\mu$-almost every $x$,
we have:
\begin{enumerate}\setlength{\itemsep}{+1.0mm}
\item $\phi$ is twice differential at $x$, and in particular, $F_{t}(x)\notin \cut x$ for every $t\in [0,1]$;
\item for every $t\in [0,1]$ the determinant $\det\, (dF_{t})_{x}$ of $dF_{t}$ at $x$ is positive;
\item $\rho_{0}(x)=\rho_{1}(F_{1}(x))\,e^{-f(F_{1}(x))+f(x)}\,\det  (dF_{1})_{x}$,
         where $\rho_{0}$ and $\rho_{1}$ are the densities of $\mu$ and of $\nu$ with respect to $m$,
        respectively.
\end{enumerate}
\end{thm}

\section{Key inequalities}\label{sec:Key inequalities}
In the present section,
we will prove the following key inequality for the proof of Theorem \ref{thm:displacement convexity}:
\begin{prop}\label{prop:Jacobian inequality}
Let $\mu,\nu\in P^{ac}_{2}(M)$ be disjointly supported,
and let $\phi$ denote the Kantorovich potential obtained in Theorem \ref{thm:McCann theorem}.
Fix a point $x \in M$.
Assume that
$\phi$ is twice differentiable at $x$,
and $\det \,(dF_{t})_{x}>0$ for every $t \in [0,1]$.
For each $t \in [0,1]$
we put
\begin{equation*}\label{eq:weighted Jacobian}
J_{t}(x):=e^{-f(F_{t}(x))+f(x)}\,\det (d F_{t})_{x}.
\end{equation*}
For $\kappa \in \mathbb{R}$,
if $\ric^{1}_{f} \geq (n-1)\kappa\,e^{\frac{-4f}{n-1}}$,
then for every $t \in (0,1)$
\begin{equation*}
J_{t}(x)^{\frac{1}{n}} \geq (1-t)\,  \beta_{\kappa,f,1-t}(F_{1}(x),x)^{\frac{1}{n}}  \, J_{0}(x)^{\frac{1}{n}}+t\,  \beta_{\kappa,f,t}(x,F_{1}(x))^{\frac{1}{n}}  \,J_{1}(x)^{\frac{1}{n}}.
\end{equation*}
\end{prop}

Throughout this section,
as in Proposition \ref{prop:Jacobian inequality},
let $\mu,\nu\in P^{ac}_{2}(M)$ be disjointly supported,
and let $\phi$ denote the associated Kantorovich potential.
Moreover,
for a fixed point $x \in M$,
we assume that
$\phi$ is twice differentiable at $x$,
and $\det \,(dF_{t})_{x}>0$ for all $t \in [0,1]$.

\subsection{Riccati inequalities}
Define a curve $\gamma:[0,1] \to M$ by $\gamma(t):=F_{t}(x)$,
and choose an orthonormal basis $\{e_{i}\}^{n}_{i=1}$ of the tangent space at $x$ with $e_{n}=\gamma'(0)/\Vert \gamma'(0)\Vert$,
where $\Vert \cdot \Vert$ is the canonical norm induced from $g$.
For each $i=1,\dots,n$,
we define a Jacobi field $E_{i}$ along $\gamma$ by $E_{i}(t):=(dF_{t})_{x}(e_{i})$ (cf. the proof of Theorem 1.7 in \cite{St3}).
For each $t \in [0,1]$
let $A(t)=(a_{ij}(t))$ be an $n \times n$ matrix determined by
\begin{equation}\label{eq:Jacobi field}
E'_{i}(t)=\sum^{n}_{j=1}\,a_{ij}(t)\,E_{j}(t).
\end{equation}
We define a function $h:[0,1]\to \mathbb{R}$ by
\begin{equation*}
h(t):=\log \,\det (dF_{t})_{x} -\int^{t}_{0}\,a_{nn}(\xi)\,d\xi.
\end{equation*}

It is well-known that
the function $h$ satisfies the following inequality of Riccati type (see e.g., (1.4), (1.9) in \cite{St3}, and (14.21) in \cite{V}):
\begin{lem}\label{lem:unweighted Riccati inequality}
For every $t \in (0,1)$
we have
\begin{equation*}
h''(t) \leq  -\frac{h'(t)^{2}}{n-1}-\ric_{g}(\gamma'(t)).
\end{equation*}
\end{lem}

We define a function $l:[0,1]\to \mathbb{R}$ by
\begin{equation*}\label{eq:weighted log Jacobian}
l(t):=h(t)-f(\gamma(t))+f(x).
\end{equation*}

For distance functions,
Wylie and Yeroshkin \cite{WY} have obtained an inequality of Riccati type that is associated with $\ric^{1}_{f}$ (see Lemma 4.1 in \cite{WY}).
By using the same method,
we have the following:
\begin{lem}\label{lem:Riccati inequality}
For every $t \in (0,1)$
we have
\begin{equation*}
\left(   e^{\frac{2f(\gamma(t))}{n-1}} \, l'(t)   \right)' \leq -e^{\frac{2 f(\gamma(t))}{n-1}} \left(      \frac{ l'(t)^{2}}{n-1}+\ric^{1}_{f}(\gamma'(t))      \right).
\end{equation*}
\end{lem}
\begin{proof}
Put $f_{x}:=f\circ \gamma$.
From Lemma \ref{lem:unweighted Riccati inequality}
we deduce
\begin{align*}
l''(t)=h''(t)-f''_{x}(t) & \leq -\frac{h'(t)^{2}}{n-1} -\left(\ric_{g}(\gamma'(t))+ f''_{x}(t)   \right)\\
                                         &   =   -\frac{l'(t)^{2}}{n-1}-\frac{2\,l'(t)\,f'_{x}(t)}{n-1} -\ric^{1}_{f}(\gamma'(t)).
\end{align*}
Hence
we have
\begin{align*}
e^{\frac{-2f_{x}(t)}{n-1}}\left(   e^{\frac{2f_{x}(t)}{n-1}} \,l'(t)    \right)'  &    =   l''(t)+\frac{2\,l'(t)\,f'_{x}(t)}{n-1} \leq -\frac{l'(t)^{2}}{n-1} -\ric^{1}_{f}(\gamma'(t)).
\end{align*}
This proves the desired inequality.
\end{proof}

\subsection{Jacobian inequalities}
We recall the following elementary comparison argument (see e.g., Theorem 14.28 in \cite{V}):
\begin{lem}\label{lem:comparison argument}
For $a>0$,
let $D:[0,a]\to \mathbb{R}$ be a non-negative continuous function
that is $C^{2}$ on $(0,a)$.
Take $\kappa \in \mathbb{R}$ and $\mathfrak{d} \geq 0$.
Assume $\kappa\, \mathfrak{d}^{2} \in (-\infty,a^{-2}\pi^{2})$.
Then $D''+\kappa\, \mathfrak{d}^{2} \,D \leq 0$ on $(0,a)$ if and only if 
for all $s_{0},s_{1} \in [0,a]$ and $\lambda \in [0,1]$,
\begin{equation*}
D\left((1-\lambda)s_{0}+\lambda s_{1}\right) \geq \frac{\mathfrak{s}_{\kappa}((1-\lambda)\,\vert s_{0}-s_{1} \vert\, \mathfrak{d})}{\mathfrak{s}_{\kappa}(\vert s_{0}-s_{1} \vert\, \mathfrak{d})} \,D(s_{0})
                                                                               + \frac{\mathfrak{s}_{\kappa}(\lambda\,\vert s_{0}-s_{1} \vert\, \mathfrak{d})}{\mathfrak{s}_{\kappa}(\vert s_{0}-s_{1} \vert\, \mathfrak{d})}\,D(s_{1}).
\end{equation*}
\end{lem}

We define a function $D:[0,1]\to \mathbb{R}$ by
\begin{equation*}
D(t):=\exp \left(\frac{l(t)}{n-1} \right).
\end{equation*}

Lemmas \ref{lem:Riccati inequality} and \ref{lem:comparison argument} yield the following concavity of the function $D$:
\begin{lem}\label{lem:pre-Jacobian inequality}
For $\kappa \in \mathbb{R}$,
if $\ric^{1}_{f} \geq (n-1)\kappa\,e^{\frac{-4f}{n-1}}$,
then for every $t \in (0,1)$
we have
\begin{equation*}
D(t) \geq  \frac{\mathfrak{s}_{\kappa}(d_{f,1-t}(F_{1}(x),x))}{\mathfrak{s}_{\kappa}(d_{f}(F_{1}(x),x))}\,D(0)+\frac{\mathfrak{s}_{\kappa}(d_{f,t}(x,F_{1}(x)))}{\mathfrak{s}_{\kappa}(d_{f}(x,F_{1}(x)))} \, D_{x}(1).
\end{equation*}
\end{lem}
\begin{proof}
We define a function $s_{f}:[0,1]\to \mathbb{R}$ by
\begin{equation*}
s_{f}(t):=\int^{t}_{0}\,e^{\frac{-2f(\gamma(\xi))}{n-1}}\,d\xi.
\end{equation*}
Put $a:=s_{f}(1)$,
and let $t_{f}:[0,a]\to [0,1]$ be the inverse function of $s_{f}$.
Define functions $\widehat{l},\, \widehat{D}:[0,a]\to \mathbb{R}$ by
\begin{equation*}
\widehat{l}:=l\circ t_{f},\quad \widehat{D}:=D \circ t_{f}.
\end{equation*}
For each $s \in (0,a)$
we see
\begin{equation}\label{eq:Riccati equation}
(n-1)\frac{\widehat{D}''(s)}{\widehat{D}(s)} =     \widehat{l}''(s)+\frac{\widehat{l}'(s)^{2}}{n-1}.
\end{equation}
We also define functions $L:[0,1]\to \mathbb{R}$ and $\widehat{L}:[0,a]\to \mathbb{R}$ by 
\begin{equation*}
L(t):=e^{\frac{2f(\gamma(t))}{n-1}}\,l'(t),\quad \widehat{L}:=L \circ t_{f}.
\end{equation*}
Note that $\widehat{l}'(s)=\widehat{L}(s)$.
From Lemma \ref{lem:Riccati inequality},
it follows that
\begin{align}\label{eq:application of Riccati inequality}
\widehat{l}''(s)   &    =    \widehat{L}'(s)=t'_{f}(s)\,L'(t_{f}(s))\\ \notag
                        &   \leq  -e^{\frac{4f\left(\gamma\left(t_{f}(s)\right)\right)}{n-1}}\left(      \frac{ l'(t_{f}(s))^{2}}{n-1}+\ric^{1}_{f}(\gamma'(t_{f}(s))) \right)\\ \notag
                        &     =   -\frac{\widehat{l}'(s)^{2}}{n-1}-e^{\frac{4f\left(\gamma\left(t_{f}(s)\right)   \right)}{n-1}}  \,\ric^{1}_{f}(\gamma'(t_{f}(s))). \notag
\end{align}
Combining (\ref{eq:Riccati equation}) and (\ref{eq:application of Riccati inequality}),
we obtain
\begin{equation*}
(n-1)\frac{\widehat{D}''(s)}{\widehat{D}(s)} \leq -e^{\frac{4f\left(\gamma\left(t_{f}(s)\right)\right)}{n-1}}  \,\ric^{1}_{f}(\gamma'(t_{f}(s)))\leq  -(n-1)\,\kappa\,d(x,y)^{2},
\end{equation*}
where $y:=F_{1}(x)$.
Therefore,
$\widehat{D}''+\kappa\, d(x,y)^{2}\,  \widehat{D} \leq 0$ on $(0,a)$.

Since the Kantorovich potential $\phi$ is twice differentiable at $x$,
the curve $\gamma$ lies in the complement of $\cut x$ in view of Theorem \ref{thm:Jacobian equation}.
In particular,
$\gamma$ is a unique minimal geodesic from $x$ to $y$,
and hence
\begin{equation*}\label{eq:twisted distance function representation}
a\,d(x,y)=d_{f}(x,y)<\tau_{f}\left(\frac{\gamma'(0)}{\Vert \gamma'(0)\Vert}\right).
\end{equation*}
By Theorem \ref{thm:cut value estimate},
$\kappa\, d(x,y)^{2} \in (-\infty,a^{-2}\pi^{2})$.
Lemma \ref{lem:comparison argument} implies that
for all $s_{0},s_{1} \in [0,a]$ and $\lambda \in [0,1]$
\begin{equation}\label{eq:use of comparison argument}
\widehat{D}\left((1-\lambda)s_{0}+\lambda s_{1}\right) \geq \frac{\mathfrak{s}_{\kappa}((1-\lambda)\,\vert s_{0}-s_{1} \vert\, d(x,y))}{\mathfrak{s}_{\kappa}(\vert s_{0}-s_{1} \vert\, d(x,y))} \,\widehat{D}(s_{0})
                                                                                          +    \frac{\mathfrak{s}_{\kappa}(\lambda\,\vert s_{0}-s_{1} \vert\, d(x,y))}{\mathfrak{s}_{\kappa}(\vert s_{0}-s_{1} \vert\, d(x,y))}\,\widehat{D}(s_{1}).
\end{equation}
For every $s\in (0,a)$
we obtain
\begin{equation*}
\widehat{D}(s) \geq  \frac{\mathfrak{s}_{\kappa}((a-s)\,d(x,y))}{\mathfrak{s}_{\kappa}(a\,d(x,y))}\,  \widehat{D}(0)+  \frac{\mathfrak{s}_{\kappa}(s\,d(x,y))}{\mathfrak{s}_{\kappa}(a\,d(x,y))} \, \widehat{D}(a)
\end{equation*}
by letting $s_{0}\to 0,\,s_{1}\to a$ and $\lambda \to s/a$ in (\ref{eq:use of comparison argument}).
For every $t\in (0,1)$
\begin{equation*}
D(t) \geq  \frac{\mathfrak{s}_{\kappa}((a-s_{f}(t))\,d(x,y))}{\mathfrak{s}_{\kappa}(a\,d(x,y))}\,D(0)+  \frac{\mathfrak{s}_{\kappa}(s_{f}(t)\,d(x,y))}{\mathfrak{s}_{\kappa}(a\,d(x,y))} \,D(1).
\end{equation*}
From the uniqueness of the geodesic $\gamma$ between $x$ and $y$,
for every $t\in [0,1]$
we see
\begin{equation*}
\left(a-s_{f}(t)\right)\,d(x,y)=d_{f,1-t}(y,x),\quad s_{f}(t)\,d(x,y)=d_{f,t}(x,y).
\end{equation*}
This completes the proof.
\end{proof}

Now,
we prove Proposition \ref{prop:Jacobian inequality}.
\begin{proof}[Proof of Proposition \ref{prop:Jacobian inequality}]
For $\kappa \in \mathbb{R}$,
we assume $\ric^{1}_{f} \geq (n-1)\kappa\,e^{\frac{-4f}{n-1}}$.
Define a function $\overline{D}:[0,1]\to \mathbb{R}$ by
\begin{equation*}
\overline{D}(t):=\exp \left( \int^{t}_{0}\,a_{nn}(\xi)\,d\xi \right),
\end{equation*}
where $a_{nn}$ is determined by (\ref{eq:Jacobi field}).
The following concavity is well-known (see e.g., (1.10) in \cite{St3}, and (14.19) in \cite{V}):
\begin{equation}\label{eq:concavity}
\overline{D}(t) \geq  (1-t)\,  \overline{D}(0)+t \, \overline{D}(1)
\end{equation}
for every $t \in (0,1)$.
By Lemma \ref{lem:pre-Jacobian inequality}, (\ref{eq:concavity}) and the H\"older inequality,
we obtain
\begin{align*}\label{eq:proof of Jacobian inequality}
J_{t}(x)^{\frac{1}{n}} 
&= D(t)^{1-\frac{1}{n}}\,\overline{D}(t)^{\frac{1}{n}}\\
&\geq (1-t)^{\frac{1}{n}}\,\left(\frac{\mathfrak{s}_{\kappa}((d_{f,1-t}(F_{1}(x),x))}{\mathfrak{s}_{\kappa}(d_{f}(F_{1}(x),x))}\right)^{1-\frac{1}{n}}\, J_{0}(x)^{\frac{1}{n}} \\ \notag
                       &\qquad \,\,\,\,    +  t^{\frac{1}{n}}\,\left(\frac{\mathfrak{s}_{\kappa}(d_{f,t}(x,F_{1}(x)))}{\mathfrak{s}_{\kappa}(d_{f}(x,F_{1}(x)))}\right)^{1-\frac{1}{n}} \, J_{1}(x)^{\frac{1}{n}}.
\end{align*}
The right hand side is equal to that of the desired one.
Therefore,
we conclude the proposition.
\end{proof}

\section{Displacement convexity}\label{sec:Displacement convexity}
In this section,
we prove Theorem \ref{thm:displacement convexity} by using Proposition \ref{prop:Jacobian inequality} along the line of the proof of Theorem 1.7 in \cite{St3}.

\subsection{Curvature bounds imply displacement convexity}
First,
we prove the implication from \ref{enum:curv cond} to \ref{enum:twisted curv} in Theorem \ref{thm:displacement convexity}.
Precisely,
we show:
\begin{prop}\label{prop:only if part}
For $\kappa \in \mathbb{R}$,
if $\ric^{1}_{f} \geq (n-1)\kappa\,e^{\frac{-4f}{n-1}}$,
then $(M,d,m)$ has $\kappa$-twisted curvature bound.
\end{prop}
\begin{proof}
Let $\mu,\nu \in P^{ac}_{2}(M)$ be disjointly supported,
and let $\phi$ denote the Kantorovich potential obtained in Theorem \ref{thm:McCann theorem}.
The map $F_{t}$ on $M$ defined as (\ref{eq:interpolation map}) gives a unique optimal coupling $\pi$ of $(\mu,\nu)$ via $\pi:=(F_{0}\times F_{1})_{\#}\mu$.
Furthermore,
it determines a unique minimal geodesic $(\mu_{t})_{t\in [0,1]}$ in $(P_{2}(M),W_{2})$ from $\mu$ to $\nu$ via $\mu_{t}:=(F_{t})_{\#}\mu_{0}$,
which lies in $P^{ac}_{2}(M)$.
Moreover,
in virtue of Theorem \ref{thm:Jacobian equation},
for a fixed $t\in (0,1)$,
the Jacobian equations 
\begin{equation}\label{eq:Jacobian equation}
\rho_{0}(x)=\rho_{1}(F_1(x))\,J_{1}(x)=\rho_{t}(F_{t}(x))\,J_{t}(x)
\end{equation}
hold for $\mu_{0}$-almost every $x\in M$,
where $\rho_{t}$ is the density of $\mu_{t}$ with respect to $m$.

For $U\in \mathcal{DC}$,
let $\varphi_{U}$ be the function defined as $\varphi_{U}(r):=r^{n}\,U(r^{-n})$,
which is non-increasing and convex.
By using (\ref{eq:Jacobian equation}) and $\mu_{t}=(F_{t})_{\#}\mu_{0}$,
by the properties of $\varphi_{U}$ and Proposition \ref{prop:Jacobian inequality},
and by using (\ref{eq:Jacobian equation}) again,
\begin{align*}
U_{m}(\mu_{t})&=\int_{M}\,U\left(\frac{\rho_{0}(x)}{J_{t}(x) }\right)\frac{J_{t}(x)}{\rho_{0}(x)}\,d\mu_{0}(x)=\int_{M}\,\varphi_{U}\left(  \left(\frac{J_{t}(x)}{\rho_{0}(x)}\right)^{\frac{1}{n}} \right)\,d\mu_{0}(x)\\
                        & \leq (1-t)\,\int_{M}\,\varphi_{U} \left(    \beta_{\kappa,f,1-t}(F_{1}(x),x)^{\frac{1}{n}} \,  \left(\frac{J_{0}(x)}{\rho_{0}(x)}\right)^{\frac{1}{n}}  \,\right)\,d\mu_{0}(x)\\ \notag
                       &\qquad \,\,\, +t\,\int_{M}\,\varphi_{U} \left(  \beta_{\kappa,f,t}(x,F_{1}(x))^{\frac{1}{n}}  \,  \left(\frac{J_{1}(x)}{\rho_{0}(x)}\right)^{\frac{1}{n}} \,\right)\,d\mu_{0}(x)\\
                        & \leq (1-t)\,\int_{M}\,\varphi_{U} \left(  \left(  \frac{     \beta_{\kappa,f,1-t}(F_{1}(x),x) }{\rho_{0}(x) }   \right)^{\frac{1}{n}}  \right)\,d\mu_{0}(x)\\ \notag
                        &\qquad \,\,\, +t\,\int_{M}\,\varphi_{U} \left(   \left(\frac{ \beta_{\kappa,f,t}(x,F_{1}(x))}{  \rho_{1}(F_{1}(x)) }  \right)^{\frac{1}{n}} \right)\,d\mu_{0}(x).
\end{align*}
Since $\pi=\left( F_{0}\times F_{1}  \right)_{\#}\mu_{0}$,
the right hand side of the above inequality is equal to that of the desired one.
We complete the proof.
\end{proof}

\begin{rem}\label{rem:Jacobian inequality for density}
Under the same setting and notation as in the above proof,
we also see the following inequality by combining Proposition \ref{prop:Jacobian inequality} and (\ref{eq:Jacobian equation}):
For each fixed $t\in (0,1)$
we have
\begin{equation}\label{eq:Jacob inequ density}
\frac{1}{\rho_t(F_t(x))^{\frac{1}{n}} }\geq (1-t)\, \left(\frac{\beta_{\kappa,f,1-t}(F_{1}(x),x)}{\rho_0(x)}\right)^{\frac{1}{n}}+t\,  \left(\frac{\beta_{\kappa,f,t}(x,F_{1}(x))}{\rho_1(F_1(x)) }\right)^{\frac{1}{n}}
\end{equation}
for $\mu_{0}$-almost every $x\in M$.
\end{rem}

\subsection{Displacement convexity implies curvature bounds}
In Theorem \ref{thm:displacement convexity},
the implication from \ref{enum:twisted curv} to \ref{enum:relaxed twisted curv} is trivial.
Hence,
to conclude the desired assertion,
it suffices to prove the one from \ref{enum:relaxed twisted curv} to \ref{enum:curv cond}.

For subsets $X,Y \subset M$ and $t \in [0,1]$,
let $Z_{t}(X,Y)$ be the set of all points $\gamma(t)$,
where $\gamma:[0,1]\to M$ is a minimal geodesic with $\gamma(0) \in X,\,\gamma(1) \in Y$.
To prove the implication from \ref{enum:relaxed twisted curv} to \ref{enum:curv cond},
we prepare the following lemma,
which claims that
the relaxed twisted curvature bound implies the inequality of Brunn-Minkowski type.
\begin{lem}\label{lem:Brunn-Minkowski type inequality}
Let $X,\,Y \subset M$ denote two disjoint, bounded Borel subsets with $m(X),\,m(Y)\in (0,\infty)$.
For $\kappa \in \mathbb{R}$,
if $(M,d,m)$ has $\kappa$-relaxed twisted curvature bound,
then for every $t\in (0,1)$
we have
\begin{align}\label{eq:Brunn-Minkowski type inequality}
m \left(  Z_{t}(X,Y)   \right)^{\frac{1}{n}} \geq (1-t)\, &\left(\inf_{(x,y)\in X \times Y}\,\beta_{\kappa,f,1-t}(y,x)^{\frac{1}{n}} \right)  \,  m(X)^{\frac{1}{n}}\\
                                                                        +\: t\, &\left(\inf_{(x,y)\in X \times Y}\,\beta_{\kappa,f,t}(x,y)^{\frac{1}{n}}   \right)\,  m(Y)^{\frac{1}{n}}. \notag
\end{align}
\end{lem}
\begin{proof}
Let $1_{X}$ and $1_{Y}$ be the characteristic functions of $X$ and of $Y$,
respectively.
We set
\begin{equation*}
\rho_{0}:=\frac{1_{X}}{m(X)},\quad \mu_{0}:=\rho_{0}\,m,\quad \rho_{1}:=\frac{1_{Y}}{m(Y)},\quad \mu_{1}:=\rho_{1}\,m.
\end{equation*}
Let $(\mu_{t})_{t\in [0,1]}$ denote a unique minimal geodesic in $(P_{2}(M),W_{2})$ from $\mu_{0}$ to $\mu_{1}$,
which lies in $P^{ac}_{2}(M)$.
Notice that
$\mu_{t}$ is supported on $Z_{t}(X,Y)$ for each $t \in (0,1)$ since it is written in the form of $\mu_{t}=(F_{t})_{\#}\mu_{0}$ for the map $F_t$ defined as (\ref{eq:interpolation map}),
and $\mu_0$ and $\mu_1$ are supported on $X$ and $Y$,
respectively.
Let $\rho_{t}$ stand for the density of $\mu_{t}$ with respect to $m$.
From the Jensen inequality
one can derive
\begin{equation*}
m \left(  Z_{t}(X,Y)   \right)^{\frac{1}{n}} \geq \int_{M}\,\rho_{t}(x)^{1-\frac{1}{n}}\,dm(x).
\end{equation*}
Since $(M,d,m)$ has $\kappa$-relaxed twisted curvature bound,
we have
\begin{align}\label{eq:proof of Brunn-Minkowski inequality}
\int_{M}\,\rho_{t}(x)^{1-\frac{1}{n}}\,dm(x) \geq (1-t)&\,\int_{M\times M}\, \rho_{0}(x)^{-\frac{1}{n}} \,\beta_{\kappa,f,1-t}(y,x)^{\frac{1}{n}}\, d\pi(x,y)\\
                                                                              +\:t&\,\int_{M\times M}\,   \rho_{1}(y)^{-\frac{1}{n}} \,\beta_{\kappa,f,t}(x,y)^{\frac{1}{n}}\, d\pi(x,y), \notag
\end{align}
where $\pi$ is a unique optimal coupling of $(\mu_{0},\mu_{1})$.
Now,
$\pi$ is supported on $X\times Y$ since it is the coupling,
and $\mu_0$ and $\mu_1$ are supported on $X$ and $Y$,
respectively.
Hence,
the right hand side of (\ref{eq:proof of Brunn-Minkowski inequality}) is bounded from below by
\begin{align*}
(1-t)\, &\left(\inf_{(x,y)\in X \times Y}\,\beta_{\kappa,f,1-t}(y,x)^{\frac{1}{n}} \right)  \, \int_{M}\,\rho_{0}(x)^{1-\frac{1}{n}}\,dm(x)\\
+\:t\, &\left(\inf_{(x,y)\in X \times Y}\,\beta_{\kappa,f,t}(x,y)^{\frac{1}{n}}   \right)\,  \int_{M}\,\rho_{1}(y)^{1-\frac{1}{n}}\,dm(y) \notag
\end{align*}
that is equal to the right hand side of (\ref{eq:Brunn-Minkowski type inequality}).
This proves the lemma.
\end{proof}

Based on Lemma \ref{lem:Brunn-Minkowski type inequality},
we can prove the implication from \ref{enum:relaxed twisted curv} to \ref{enum:curv cond} in Theorem \ref{thm:displacement convexity} as follows:
\begin{prop}\label{prop:if part}
For $\kappa \in \mathbb{R}$,
if $(M,d,m)$ has $\kappa$-relaxed twisted curvature bound,
then $\ric^{1}_{f} \geq (n-1)\kappa\,e^{\frac{-4f}{n-1}}$.
\end{prop}
\begin{proof}
Fix $x\in M$ and $v\in U_{x}M$.
For $\epsilon>0$,
let $\gamma:(-\epsilon,\epsilon)\to M$ be the geodesic with $\gamma(0)=x,\,\gamma'(0)=v$.
Take $\delta \in (0,\epsilon)$ and $\eta \in (0,\delta)$.
For $y \in M$,
we denote by $B_{\eta}(y)$ the open geodesic ball of radius $\eta$ centered at $y$.
We set $X:=B_{\eta}(\gamma(-\delta))$ and $Y:=B_{\eta}(\gamma(\delta))$,
which are disjoint.
From Lemma \ref{lem:Brunn-Minkowski type inequality}
we deduce
\begin{align*}
m \left(  Z_{\frac{1}{2}}(X,Y)\right)^{\frac{1}{n}} \geq \frac{1}{2}\,\left(\inf_{(x,y)\in X \times Y}\,\beta_{\kappa,f,\frac{1}{2}}(y,x)^{\frac{1}{n}} \right)  \,  &m(X)^{\frac{1}{n}}\\
                                                               +\: \frac{1}{2}\,\left(\inf_{(x,y)\in X \times Y}\,\beta_{\kappa,f,\frac{1}{2}}(x,y)^{\frac{1}{n}}   \right)\,  &m(Y)^{\frac{1}{n}}.
\end{align*}
By letting $\eta \to 0$ in the above inequality,
\begin{align}\label{eq:use of Brunn-Minkowski type inequality}
\liminf_{\eta \to 0}\,\left(\frac{m \left(    Z_{\frac{1}{2}}(X,Y)  \right)}{\omega_{n}\,\eta^{n}}\right)^{\frac{1}{n}} \geq \frac{1}{2}\,&\left(e^{-f(\gamma(-\delta))}\,\beta_{\kappa,f,\frac{1}{2}}(\gamma(\delta),\gamma(-\delta))\right)^{\frac{1}{n}}\\ \notag
                                                                                                                                                                            +\: \frac{1}{2}\,&\left(e^{-f(\gamma(\delta))}\,\beta_{\kappa,f,\frac{1}{2}}(\gamma(-\delta),\gamma(\delta))\right)^{\frac{1}{n}},
\end{align}
where $\omega_{n}$ denotes the volume of the unit ball in $\mathbb{R}^{n}$.

By the definition of the function $d_{f,t}$,
we see
\begin{align*}
d_{f,\frac{1}{2}}(\gamma(\delta),\gamma(-\delta))&=\int^{\delta}_{0}\,e^{\frac{-2f(\gamma(\xi))}{n-1}}\,d\xi,\\
d_{f,\frac{1}{2}}(\gamma(-\delta),\gamma(\delta))&=\int^{0}_{-\delta}\,e^{\frac{-2f(\gamma(\xi))}{n-1}}\,d\xi,\\
d_{f}(\gamma(-\delta),\gamma(\delta))&=\int^{\delta}_{-\delta}\,e^{\frac{-2f(\gamma(\xi))}{n-1}}\,d\xi.
\end{align*}
Hence,
the Taylor series of $\beta_{\kappa,f,\frac{1}{2}}(\gamma(\delta),\gamma(-\delta))$ and $\beta_{\kappa,f,\frac{1}{2}}(\gamma(-\delta),\gamma(\delta))$ with respect to $\delta$ at $0$ are given as
\begin{align*}
&\quad \,\,\beta_{\kappa,f,\frac{1}{2}}(\gamma(\delta),\gamma(-\delta))\\
&=1-g((\nabla f)_{x},v)\, \delta+\left(  (n-1)\,\kappa\,e^{\frac{-4f(x)}{n-1}} +\frac{n-2}{n-1} \,g((\nabla f)_{x},v)^{2} \right) \,\frac{\delta^{2}}{2}+O(\delta^{3}),\\
&\quad \,\,\beta_{\kappa,f,\frac{1}{2}}(\gamma(-\delta),\gamma(\delta))\\
&=1+g((\nabla f)_{x},v)\, \delta +\left( (n-1)\,\kappa\,e^{\frac{-4f(x)}{n-1}}+\frac{n-2}{n-1} \,g((\nabla f)_{x},v)^{2} \right) \,\frac{\delta^{2}}{2}+O(\delta^{3}).
\end{align*}
On the other hand,
\begin{align*}
e^{-f(\gamma(-\delta))+f(x)}&=1+g((\nabla f)_{x},v)\,\delta+\left( g((\nabla f)_{x},v)^{2}-\Hess f(v,v)  \right)\frac{\delta^{2}}{2} +O(\delta^{3}),\\
e^{-f(\gamma(\delta))+f(x)}&=1-g( (\nabla f)_{x},v)\,\delta+\left( g((\nabla f)_{x},v)^{2}-\Hess f (v,v)  \right)\frac{\delta^{2}}{2} +O(\delta^{3}).
\end{align*}
Substituting these series into (\ref{eq:use of Brunn-Minkowski type inequality}),
we have
\begin{align}\label{eq:Taylor series}
&\quad \,\,\liminf_{\eta \to 0}\,\frac{m \left(    Z_{\frac{1}{2}}(X,Y)  \right)}{\omega_{n}\,\eta^{n}}\\ \notag
&\geq e^{-f(x)}+e^{-f(x)}\left((n-1)\,\kappa\,e^{\frac{-4f(x)}{n-1}}   -\Hess f(v,v)+\frac{g((\nabla f)_{x},v)^{2}}{1-n}\,  \right)\, \frac{\delta^{2}}{2}+O(\delta^{3}).
\end{align}

Recall the following inequality (see e.g., the proof of Theorem 1.7 in \cite{St3}, and that of Theorem 1.2 in \cite{O1}):
\begin{equation}\label{eq:Ricci curvature representation}
\limsup_{\eta \to 0}\,\frac{m \left(    Z_{\frac{1}{2}}(X,Y)  \right)}{\omega_{n}\,\eta^{n}} \leq e^{-f(x)}\,\left( 1+\ric_{g}(v)\,\frac{\delta^{2}}{2}\,\right)+O(\delta^{3}).
\end{equation}
Comparing (\ref{eq:Taylor series}) with (\ref{eq:Ricci curvature representation}),
we obtain
\begin{equation*}
\ric_{g}(v)\geq (n-1)\,\kappa\,e^{\frac{-4f(x)}{n-1}}-\Hess f(v,v)+\frac{g((\nabla f)_{x},v)^{2}}{1-n};
\end{equation*}
in particular,
$\ric^{1}_{f}(v) \geq (n-1)\kappa\,e^{\frac{-4f(x)}{n-1}}$.
This completes the proof.
\end{proof}

We are now in a position to conclude Theorem \ref{thm:displacement convexity}.

\begin{proof}[Proof of Theorem \ref{thm:displacement convexity}]
The implication from \ref{enum:curv cond} to \ref{enum:twisted curv} directly follows from Proposition \ref{prop:only if part}.
The one from \ref{enum:twisted curv} to \ref{enum:relaxed twisted curv} is trivial since the $\kappa$-relaxed twisted curvature bound requires that (\ref{eq:lower twisted curvature bound}) holds only for $H\in \mathcal{DC}$.
The one from \ref{enum:relaxed twisted curv} to \ref{enum:curv cond} is a direct consequence of Proposition \ref{prop:if part}.
Thus,
we complete the proof of Theorem \ref{thm:displacement convexity}.
\end{proof}

\subsection{Interpolation inequalities}\label{sec:Interpolation inequalities}
We end this section with the summary of interpolation inequalities under the curvature condition (\ref{eq:Ricci curvature assumption}) motivated by Lemma \ref{lem:Brunn-Minkowski type inequality} in the above subsection.

We begin with the so-called $p$-mean inequality.
Let $t\in (0,1)$ and $a,b \in [0,\infty)$.
For $p \in \mathbb{R}\setminus \{0\}$,
the \textit{$p$-mean} is defined as follows:
\begin{equation*}
\mathcal{M}^{p}_{t}(a,b):=\left((1-t)\,a^{p}+t\,b^{p}\right)^{\frac{1}{p}}
\end{equation*}
if $ab\neq 0$,
and $\mathcal{M}^{p}_{t}(a,b):=0$ if $ab=0$.
As the limits,
it is defined as
\begin{equation*}
\mathcal{M}^{0}_{t}(a,b):=a^{1-t}\,b^{t},\quad \mathcal{M}^{\infty}_{t}(a,b):=\max \{a,b\},\quad  \mathcal{M}^{-\infty}_{t}(a,b):=\min \{a,b\}.
\end{equation*}
We possess the following (cf. Corollary 1.1 in \cite{CMS}, Corollary 9.1 in \cite{O1} and Theorem 19.18 in \cite{V}):
\begin{cor}\label{cor:p mean}
For $i=0,1$,
let $\psi_{i}:M \to \mathbb{R}$ denote non-negative, integrable functions.
Let $X,\,Y \subset M$ denote disjoint, bounded Borel subsets with $\supp \psi_{0} \subset X,\,\supp \psi_{1} \subset Y$.
Let $\psi:M\to \mathbb{R}$ be a non-negative function.
For $t \in (0,1)$ and $p \geq -1/n$,
we assume that
for all $(x,y) \in X \times Y$ and $z \in Z_{t}(\{x\},\{y\})$,
we have
\begin{equation*}
\psi(z) \geq \mathcal{M}^{p}_{t}\left(  \frac{\psi_{0}(x)}{\beta_{\kappa,f,1-t}(y,x)}, \frac{\psi_{1}(y)}{\beta_{\kappa,f,t}(x,y)}\right).
\end{equation*}
For $\kappa \in \mathbb{R}$,
if $\ric^{1}_{f} \geq (n-1)\kappa\,e^{\frac{-4f}{n-1}}$,
then we have
\begin{equation*}
\int_{M} \,\psi\,dm \geq \mathcal{M}^{\frac{p}{1+np}}_{t}\left(  \int_{M} \,\psi_{0}\,dm,\int_{M} \,\psi_{1}\,dm    \right).
\end{equation*}
Here we set $p/(1+np):=-\infty$ for $p=-1/n$.
\end{cor}
Theorem 19.18 in \cite{V} states that
for $K \in \mathbb{R}$ and $N \in [n,\infty)$,
the curvature condition (\ref{eq:constant Ricci curvature assumption}) implies an inequality of Pr\'ekopa-Leindler type.
One can prove Corollary \ref{cor:p mean}
only by replacing the role of Theorem 19.4 in \cite{V}
with that of Proposition \ref{prop:Jacobian inequality} (or rather the inequality (\ref{eq:Jacob inequ density}) in Remark \ref{rem:Jacobian inequality for density}) in the proof.
We omit the proof.

As the special case of $p=0$ in Corollary \ref{cor:p mean},
we obtain the following inequality of Pr\'ekopa-Leindler type (cf. Corollary 1.2 in \cite{CMS} and Corollary 9.2 in \cite{O1}):
\begin{cor}\label{cor:Prekopa-Leindler}
For $i=0,1$,
let $\psi_{i},X,Y,\psi$ be as in Corollary \ref{cor:p mean}.
For $t \in (0,1)$,
we assume that
for all $(x,y) \in X \times Y$ and $z \in Z_{t}(\{x\},\{y\})$,
\begin{equation*}
\psi(z) \geq  \left(  \frac{\psi_{0}(x)}{\beta_{\kappa,f,1-t}(y,x)} \right)^{1-t}\left(  \frac{\psi_{1}(y)}{\beta_{\kappa,f,t}(x,y)}  \right)^{t}.
\end{equation*}
For $\kappa \in \mathbb{R}$,
if $\ric^{1}_{f} \geq (n-1)\kappa\,e^{\frac{-4f}{n-1}}$,
then we have
\begin{equation*}
\int_{M} \,\psi\,dm \geq \left(  \int_{M} \,\psi_{0}\,dm  \right)^{1-t}\left(  \int_{M} \,\psi_{1}\,dm   \right)^{t}.
\end{equation*}
\end{cor}

Letting $p \to -1/n$ in Corollary \ref{cor:p mean} yields the following inequality of Borel-Branscamp-Lieb type (cf. Main Theorem in \cite{CMS} and Theorem 1.1 in \cite{O1}):
\begin{cor}\label{cor:BBL}
For $i=0,1$,
let $\psi_{i},X,Y,\psi$ be as in Corollary \ref{cor:p mean}.
We suppose $\int_{M} \,\psi_{0}\,dm=\int_{M} \,\psi_{1}\,dm=1$.
For $t \in (0,1)$,
we assume that
for all $(x,y) \in X \times Y$ and $z \in Z_{t}(\{x\},\{y\})$,
we have
\begin{equation*}
\frac{1}{\psi(z)^{\frac{1}{n}}} \leq (1-t)\left(  \frac{\beta_{\kappa,f,1-t}(y,x)}{\psi_{0}(x)} \right)^{\frac{1}{n}}+t\left( \frac{\beta_{\kappa,f,t}(x,y)}{\psi_{1}(y)}  \right)^{\frac{1}{n}}.
\end{equation*}
For $\kappa \in \mathbb{R}$,
if $\ric^{1}_{f} \geq (n-1)\kappa\,e^{\frac{-4f}{n-1}}$,
then we have $\int_{M} \,\psi\,dm \geq 1$.
\end{cor}

Corollary \ref{cor:BBL} leads to the following inequality of Brunn-Minkowski type (cf. Corollary 9.3 in \cite{O1} and Theorem 18.5 in \cite{V}):
\begin{cor}\label{cor:Brunn-Minkowski inequality}
Let $X,\,Y \subset M$ denote two disjoint, bounded Borel subsets with $m(X),\,m(Y)\in (0,\infty)$.
For $\kappa \in \mathbb{R}$,
if $\ric^{1}_{f} \geq (n-1)\kappa\,e^{\frac{-4f}{n-1}}$,
then for every $t\in (0,1)$
we have
\begin{align}\label{eq:Brunn-Minkowski inequality}
m \left(  Z_{t}(X,Y)   \right)^{\frac{1}{n}} \geq (1-t)\, &\left(\inf_{(x,y)\in X \times Y}\,\beta_{\kappa,f,1-t}(y,x)^{\frac{1}{n}} \right)  \,  m(X)^{\frac{1}{n}}\\
                                                                        +\: t\, &\left(\inf_{(x,y)\in X \times Y}\,\beta_{\kappa,f,t}(x,y)^{\frac{1}{n}}   \right)\,  m(Y)^{\frac{1}{n}}. \notag
\end{align}
\end{cor}
We can prove (\ref{eq:Brunn-Minkowski inequality}) by applying Corollary \ref{cor:BBL} to the functions
\begin{equation*}
\psi_{0}:=\frac{1_{X}}{m(X)},\quad \psi_{1}:=\frac{1_{Y}}{m(Y)},\quad \psi:=c^{-n}\, 1_{Z_{t}(X,Y)},
\end{equation*}
where $c$ denotes the right hand side of (\ref{eq:Brunn-Minkowski inequality}).
One can also derive (\ref{eq:Brunn-Minkowski inequality}) from Theorem \ref{thm:displacement convexity} and Lemma \ref{lem:Brunn-Minkowski type inequality}.

\section{Functional inequalities}\label{sec:Applications}
This last section is devoted to the discussion on functional inequalities under the curvature condition (\ref{eq:Ricci curvature assumption}).
For $\kappa \in \mathbb{R}$,
let $\mathfrak{c}_{\kappa}:=\mathfrak{s}'_{\kappa}$.
For $x,y\in M$ with $x\neq y$,
we define $\mathsf{b}_{\kappa,f}(x,y)$ and $\mathfrak{b}_{\kappa,f}(x,y)$ as follows:
\begin{equation*}\label{eq:first asymptotic coefficient}
\mathsf{b}_{\kappa,f}(x,y):=\left(  \frac{e^{\frac{-2f(x)}{n-1}}\,d(x,y)}{\mathfrak{s}_{\kappa}(d_{f}(x,y))} \right)^{n-1},\quad \mathfrak{b}_{\kappa,f}(x,y):=\frac{n-1}{n} \left(\frac{e^{\frac{-2f(x)}{n-1}}\,d(x,y)\,\mathfrak{c}_{\kappa}(d_{f}(x,y))}{\mathfrak{s}_{\kappa}(d_{f}(x,y))}-1 \right)
\end{equation*}
if $d_{f}(x,y)\in (0,C_{\kappa})$;
otherwise,
$\mathsf{b}_{\kappa,f}(x,y):=\infty$ and $\mathfrak{b}_{\kappa,f}(x,y):=\infty$.

\begin{rem}
In the unweighted case of $f=0$,
similarly to the twisted coefficient $\beta_{\kappa,f,t}(x,y)$,
one can define $\mathsf{b}_{\kappa,f}(x,y)$ and $\mathfrak{b}_{\kappa,f}(x,y)$ for $x=y$ as the limits $1$ and $0$,
respectively (cf. Remark \ref{rem:disjoint}).
\end{rem}

We check the following basic properties of $\mathsf{b}_{\kappa,f}(x,y)$ and $\mathfrak{b}_{\kappa,f}(x,y)$:
\begin{lem}\label{eq:asymptotic behavior}
Let $\kappa \in \mathbb{R}$.
We take $x,y \in M$ with $x\neq y$ and $d_{f}(x,y)\in (0,C_{\kappa})$.
We assume $y\notin \cut x$.
Then by letting $t\to 0$ we have
\begin{align}\label{eq:first asymptotic behavior}
\beta_{\kappa,f,t}(x,y) &\to \mathsf{b}_{\kappa,f}(x,y),\\ \label{eq:second asymptotic behavior}
\frac{1-\beta_{\kappa,f,1-t}(y,x)^{\frac{1}{n}}}{t}&\to \mathfrak{b}_{\kappa,f}(x,y).
\end{align}
\end{lem}
\begin{proof}
First,
we show (\ref{eq:first asymptotic behavior}).
Since the point $y$ does not belong to $\cut x$,
there exists a unique minimal geodesic $\gamma:[0,1]\to M$ from $x$ to $y$.
We define a function $s_{f}:[0,1]\to \mathbb{R}$ by
\begin{equation*}
s_{f}(t):=\int^{t}_{0}\,e^{\frac{-2f(\gamma(\xi))}{n-1}}\,d\xi.
\end{equation*}
The uniqueness of $\gamma$ tells us that
for every $t\in [0,1]$
we have $d_{f,t}(x,y)=s_{f}(t)\,d(x,y)$.
This implies
\begin{equation*}
\frac{\mathfrak{s}_{\kappa}(d_{f,t}(x,y))}{t\,\mathfrak{s}_{\kappa}(d_{f}(x,y))} \to \frac{s'_{f}(0)\,d(x,y)}{\mathfrak{s}_{\kappa}(d_{f}(x,y))}=\frac{e^{\frac{-2f(x)}{n-1}}\,d(x,y)}{\mathfrak{s}_{\kappa}(d_{f}(x,y))}
\end{equation*}
as $t\to 0$.
We obtain (\ref{eq:first asymptotic behavior}).

We next show (\ref{eq:second asymptotic behavior}).
For a unique minimal geodesic $\overline{\gamma}:[0,1]\to M$ from $y$ to $x$,
we define a function $\overline{s}_{f}:[0,1]\to \mathbb{R}$ as
\begin{equation*}
\overline{s}_{f}(t):=\int^{t}_{0}\,e^{\frac{-2f(\overline{\gamma}(\xi))}{n-1}}\,d\xi.
\end{equation*}
The uniqueness of $\overline{\gamma}$ implies that
$d_{f,t}(y,x)=\overline{s}_{f}(t)\,d(y,x)$ for every $t\in [0,1]$.
Define $G_{f}:(0,1]\to \mathbb{R}$ by
\begin{equation*}
G_{f}(t):=\beta_{\kappa,f,t}(y,x)^{\frac{1}{n}}=\left(\frac{\mathfrak{s}_{\kappa}(\overline{s}_{f}(t)\,d(y,x))}{t\,\mathfrak{s}_{\kappa}(d_{f}(y,x))}\right)^{1-\frac{1}{n}},
\end{equation*}
here we define $G_f(1)$ as the limit $1$.
From direct computations
we deduce
\begin{equation*}
G'_{f}(1)=\frac{n-1}{n} \left(\frac{\overline{s}'_{f}(1)\,d(x,y)\,\mathfrak{c}_{\kappa}(d_{f}(x,y))}{\mathfrak{s}_{\kappa}(d_{f}(x,y))}-1 \right)=\mathfrak{b}_{\kappa,f}(x,y).
\end{equation*}
This proves (\ref{eq:second asymptotic behavior}).
\end{proof}

For a non-negative Lipschitz function $\rho$ on $M$ with $\int_{M}\,\rho\,dm=1$,
set $\mu:=\rho \,m$.
The \textit{generalized Fisher information $I_{m}(\mu)$ of $\mu$} is defined as
\begin{equation*}
I_{m}(\mu):=\int_{M}\,\frac{ \bigl\Vert   \nabla \rho^{1-\frac{1}{n}} \bigl \Vert^{2}}{\rho}\,dm.
\end{equation*}

Recall the following fact concerning the derivative of the R\'enyi entropy $H_{m}$ (see e.g., Theorem 20.1 and (20.8) in \cite{V}):
\begin{prop}\label{prop:lower differential of energy}
For $i=0,1$,
let $\rho_{i}:M\to \mathbb{R}$ be non-negative Lipschitz functions with $\int_{M}\,\rho_{i} \,dm=1$.
We assume that
$\mu:=\rho_{0} \,m$ and $\nu:=\rho_{1}\,m$ belong to $P^{ac}_2(M)$.
Then for a unique minimal geodesic $(\mu_{t})_{t \in [0,1]}$ in $(P_{2}(M),W_{2})$ from $\mu$ to $\nu$,
we have
\begin{equation}\label{eq:lower differential of energy}
\liminf_{t\to 0}\,\frac{H_{m}(\mu_{t})-H_{m}(\mu)}{t}\geq -\sqrt{I_{m}(\mu)}\:W_{2}(\mu,\nu).
\end{equation}
\end{prop}

Using Theorem \ref{thm:displacement convexity} and Proposition \ref{prop:lower differential of energy},
we prove the following:
\begin{prop}\label{prop:pre HWI inequality}
For $i=0,1$,
let $\rho_{i}:M\to \mathbb{R}$ be non-negative Lipschitz functions with $\int_{M}\,\rho_{i} \,dm=1$.
We assume that
$\mu:=\rho_{0} \,m$ and $\nu:=\rho_{1}\,m$ are disjointly supported,
and belong to $P^{ac}_2(M)$.
For $\kappa \in \mathbb{R}$,
if $\ric^{1}_{f} \geq (n-1)\kappa\,e^{\frac{-4f}{n-1}}$,
then we have
\begin{align}\label{eq:pre HWI}
H_{m}(\mu) \leq \sqrt{I_{m}(\mu)} W_{2}(\mu,\nu)&+n\, \int_{M\times M}\,\rho_{0}(x)^{-\frac{1}{n}}\, \mathfrak{b}_{\kappa,f}(x,y)  \,d\pi(x,y)\\ \notag
                   & -    n\, \int_{M\times M}\,\rho_{1}(y)^{-\frac{1}{n}}\,\left(\mathsf{b}_{\kappa,f}(x,y)^{\frac{1}{n}}-1\right)\,d\pi(x,y)\\ \notag
                   &-    n\, \int_{M\times M}\,\left(\rho_{1}(y)^{-\frac{1}{n}}-1\right)\,d\pi(x,y),
\end{align}
where $\pi$ is a unique optimal coupling of $(\mu,\nu)$.
\end{prop}
\begin{proof}
By Theorem \ref{thm:displacement convexity},
$(M,d,m)$ has $\kappa$-relaxed twisted curvature bound.
It follows that
\begin{align*}
H_{m}(\mu_{t})\leq n-(1-t)\,n\,&\int_{M\times M}\,\rho_{0}(x)^{-\frac{1}{n}}\,\beta_{\kappa,f,1-t}(y,x)^{\frac{1}{n}}\,d\pi(x,y)\\
                                -\:t\,n\,&\int_{M\times M}\,\rho_{1}(y)^{-\frac{1}{n}}\,\beta_{\kappa,f,t}(x,y)^{\frac{1}{n}}\,d\pi(x,y),
\end{align*}
where $(\mu_{t})_{t\in [0,1]}$ is a unique minimal geodesic in $(P_{2}(M),W_{2})$ from $\mu$ to $\nu$.
This leads to
\begin{align*}
\frac{H_{m}(\mu_{t})-H_{m}(\mu)}{t} &\leq n\, \int_{M\times M}\,\rho_{0}(x)^{-\frac{1}{n}}\,   \frac{1-\beta_{\kappa,f,1-t}(y,x)^{\frac{1}{n}}}{t} \,d\pi(x,y)\\
                                                         &   +   n\, \int_{M\times M}\,\rho_{0}(x)^{-\frac{1}{n}}\,  \left(\beta_{\kappa,f,1-t}(y,x)^{\frac{1}{n}}-1\right)\,d\pi(x,y)\\
                                                         &   -    n\, \int_{M\times M}\,\rho_{1}(y)^{-\frac{1}{n}}\,\left(\beta_{\kappa,f,t}(x,y)^{\frac{1}{n}}-1\right)\,d\pi(x,y)\\
                                                         &   -    n\, \int_{M\times M}\,\left(\rho_{1}(y)^{-\frac{1}{n}}-1\right)\,d\pi(x,y)-H_{m}(\mu).
\end{align*}
Let $F_{1}$ be the map defined as (\ref{eq:interpolation map}),
which determines the unique optimal coupling of $(\mu,\nu)$.
Remark that
for $\mu$-almost every $x \in M$
we have $F_{1}(x)\notin \cut x$ by Theorem \ref{thm:Jacobian equation};
in particular,
Theorem \ref{thm:cut value estimate} implies $d_{f}(x,F_{1}(x)) \in (0,C_{\kappa})$.
By using $\pi =(F_{0}\times F_{1})_{\#}\mu$ and Lemma \ref{eq:asymptotic behavior},
\begin{align*}\label{eq:upper differential of energy}
\limsup_{t\to 0}\frac{H_{m}(\mu_{t})-H_{m}(\mu)}{t} &\leq n \int_{M\times M}\rho_{0}(x)^{-\frac{1}{n}} \mathfrak{b}_{\kappa,f}(x,y)\, d\pi(x,y)\\ \notag
                                                                                     &  -    n \int_{M\times M} \rho_{1}(y)^{-\frac{1}{n}}\left(\mathsf{b}_{\kappa,f}(x,y)^{\frac{1}{n}}-1\right) d\pi(x,y)\\ \notag
                                                                                     &  -    n \int_{M\times M}\left(\rho_{1}(y)^{-\frac{1}{n}}-1\right)d\pi(x,y)-H_{m}(\mu).
\end{align*}
Comparing this inequality with (\ref{eq:lower differential of energy}),
we arrive at the desired one.
\end{proof}

\begin{rem}
Under the curvature condition (\ref{eq:constant Ricci curvature assumption}) for $K\in \mathbb{R}$ and $N\in [n,\infty)$,
it is well-known that
a similar inequality to (\ref{eq:pre HWI}) holds without the disjointness of $\mu$ and $\nu$.
Moreover,
under the setting of $m\in P_2(M)$,
letting $\rho_{0}=1$ or $\rho_1=1$ leads to several functional inequalities such as the HWI inequality, the Logarithmic Sobolev inequality, and the transport energy inequality (cf. Theorems 20.10, 21.7, 22.37, and Corollary 22.39 in \cite{V}).
In our case,
even if $m\in P_2(M)$,
we can not take $\rho_{0}=1$ or $\rho_1=1$ because of the disjointness of $\mu$ and $\nu$.
Thus,
it seems to be difficult to obtain such functional inequalities under the curvature condition (\ref{eq:Ricci curvature assumption}).
\end{rem}

\subsection*{{\rm Acknowledgements}}
The author would like to thank Professor Shin-ichi Ohta for his useful comments.
He is also grateful to the anonymous referee for careful reading and valuable suggestions.
The author gratefully acknowledges support by the European Union through the ERC-AdG ``RicciBounds" for Professor Karl-Theodor Sturm.
He is also supported in part by JSPS Grant-in-Aid for Scientific Research on Innovative Areas ``Discrete Geometric Analysis for Materials Design" (17H06460).


\end{document}